\documentclass[a4paper,10pt]{amsart}
\usepackage{amsfonts,amsmath,latexsym,amssymb} 
\usepackage{amsthm}                

\newtheorem{thm}{Theorem}[section]
\newtheorem*{thm*}{Theorem}

\newtheorem{lemma}[thm]{Lemma}

\theoremstyle{definition}

\newtheorem{remark}[thm]{Remark}

\numberwithin{equation}{section}

\title{Rearrangements and Leibniz-type rules of mean oscillations}
\author{Zolt\'an L\'eka}

\address{Royal Holloway, University of London \\ Egham Hill \\ Egham \\ Surrey \\ TW20 0EX \\ United Kingdom}
\email{zoltan.leka@rhul.ac.uk}

\thanks{This study was supported by the Marie Curie IF Fellowship, Project 'Moments', Grant no. 653943 and by the Hungarian Scientific Research Fund (OTKA) grant
no. K104206.}


\begin{document}

\begin{abstract}
      We shall prove a rearrangement inequality in probability measure spaces in order to obtain 
      sharp Leibniz-type rules of mean oscillations in $L^p$-spaces and rearrangement invariant Banach function spaces.      
\end{abstract}

\maketitle

%
%
%
%
    
   \section{Introduction}
   
       Rearrangements and rearrangement inequalities are powerful tools in functional analysis.  
       One can find applications, for instance, in symmetry problems in the calculus of variations, 
       interpolation theory or matrix analysis as well, see e.g. \cite{LL}, \cite{BS} and \cite{Bhatia}.
       
       Our aim is to provide a rearrangement inequality, which seems to have been unnoticed so far,
       in the style of the classical Hardy--Littlewood inequality. We shall apply this result to offer
       a simple new proof of sharp Leibniz-type rules of mean oscillations (or dispersions around expected values)
       in $L^p$-spaces and an extension to rearrangement invariant function spaces. 
       
       In general, Leibniz-type rules quantify the seminorms of products in function or operator algebras in terms of 
       the (semi)norms of their factors. To be a bit more precise, one may consider inequalities 
        $$ \|fg\|_Z \lesssim \|g\|_{X_1} \|f\|_{Y_1}  + \|g\|_{X_2} \|f\|_{Y_2}, $$
       for all $f,g$ in the space $Z,$ with appropriate (semi)norms $X_1, X_2, Y_1, Y_2.$ 
       However, to determine the sharp constant of the right-hand side, or to prove that it is finite, depending
       on the spaces $X_1, X_2, Y_1, Y_2$, can be a challenging problem. 
       
       Sharp Leibniz inequalities or Leibniz seminorms have appeared in M. Rieffel's fundamental 
       studies of non-commutative (quantum) metric spaces, see e.g. \cite{R0}, \cite{R1}, \cite{R2}, \cite{RL}.
       Briefly, we say that a seminorm $L$ is Leibniz on a unital algebra $(\mathcal{A}, \|\cdot\|)$ if it vanishes on the unit element 
       of $\mathcal{A}$ and $L(ab) \leq \|b\| L(a) + \|a\| L(b)$ is satisfied for all $a, b \in \mathcal{A}.$
       For instance, if $(X, d)$ denote a compact metric space,
       $$ \mbox{Lip}(f) = \sup \left\{ {|f(x) - f(y)| \over d(x,y)} \colon x, y \in X, x \neq y  \right\},$$
       the Lipschitz constant of any continuous function $f \colon X \rightarrow \mathbb{C},$ 
       defines such a seminorm on the algebra $C(X)$ endowed with the usual sup norm
       (the case $\mbox{Lip}(f)=\infty$ may happen).
       Interestingly, one can simply recover the underlying metric $d$ on $X$ through Lip.
       In general, metric data of non-commutative compact $C^*$-metric spaces can be encoded 
       by Leibniz seminorms satisfying  some crucial analytic properties as well.
       We recall that very natural sources of such seminorms are given by first-order differential calculi, or inner derivations while others are
       arising from ergodic actions of compact groups, see e.g. \cite[Section 2]{RL}.
       A simple (but non-trivial) example is the standard deviation defined in ordinary on non-commutative probability spaces \cite{R2}.    
       From a broader perspective, Dirichlet forms on real $L^2$-spaces define Leibniz seminorms as well.
       Indeed, if $\mathcal{E}$ is a Dirichlet form over the domain $D(\mathcal{E}),$ one has 
        $$ \sqrt{\mathcal{E}(fg,fg)} \leq \|f\|_\infty \sqrt{\mathcal{E}(g,g)} + \|g\|_\infty \sqrt{\mathcal{E}(f,f)}$$
       for all $f, g \in L^\infty \cap D(\mathcal{E}),$ see \cite[Theorem 1.4.2]{Fuk}, \cite[Corollary 3.3.2]{BH}.
       In fact, any Dirichlet form can be represented as a quadratic form associated to a
       closable derivation, which can serve as a direct link to Leibniz-type inequalities, see \cite{CS} and the references therein.
      
       On the other hand, we have to admit that the Kato--Ponce inequalities, concerning the fractional Laplacian, are likely to be the most well-known
       Leibniz-type inequalities. We just recall the result in the following form:            
       let $(-\Delta)^\alpha$ be the fractional Laplacian defined as the Fourier multiplier
       $$ \widehat{(- \Delta)^\alpha f}(\xi) = |\xi|^{2\alpha} \hat{f}(\xi), \quad \xi \in \mathbb{R}^n, $$
       for any $f$ in the Schwartz space $\mathcal{S}(\mathbb{R}^n).$
       If $1 < r, p_1,q_1,p_2,q_2 < \infty$ such that ${1 \over r} = {1 \over p_1} + {1 \over q_1} = {1 \over p_2} + {1 \over q_2}$
       and $0 < \alpha \leq 1,$ one has, for all $f,g \in \mathcal{S}(\mathbb{R}^n),$
        $$ \|(-\Delta)^\alpha (fg)\|_r \leq C( \|g\|_{p_1} \|(-\Delta)^\alpha f\|_{q_1} 
        + \|f\|_{p_2} \|(-\Delta)^\alpha g\|_{q_2}),$$
       where $C = C_{n,\alpha, p_1, q_1, p_2, q_2, r} > 0$ is a constant depending
       only on $(n,\alpha,  p_1, q_1, p_2, q_2, r).$ Nowadays, the Kato--Ponce inequalities
       have been extensively studied and have a large literature. We refer the reader
       to \cite{G} and \cite{MS}, for instance.
   
       In this paper, we shall provide Leibniz-type rules of mean oscillations via rearrangement inequalities.
       We would like to convince the reader that rearrangements of functions naturally appear when we discuss 
       these inequalities. The present paper is a continuation of the recent ones \cite{BeL}, \cite{LJmaa}, \cite{LMia}
       and offers a different view on our earlier results
       with an extension to rearrangement invariant function spaces. 
       
       Given a probability space $(\Omega, \mathcal{F}, \mu),$
       suppose $f, g \in L^\infty(\Omega, \mu)$ and $h \in L^1(\Omega, \mu)$ are real-valued $\mu$-measurable functions.
       First, we shall prove a rearrangement inequality in Theorem 3.2 below. There the mean-zero condition 
        $$\int_\Omega g \: d\mu = 0$$ turns out to be crucial to get 
        $$ \iint\limits_{\Omega \times \Omega} (f(x)+f(y))(g(x)-g(y))h(y)  \: d\mu(x)d\mu(y)\leq  2 \int_0^\infty f^*(y) g^*(y) h^*(y) \: dy$$
       with $f^*, g^*$ and $h^*,$ the decreasing rearrangement of $f,g$ and $h,$ respectively.
       Moreover, we shall provide an instant application 
       by proving a Leibniz-type result of mean oscillations that first appeared in \cite{LMia}.
       Indeed, setting $\displaystyle f_\Omega := \int_\Omega f \: d \mu,$ for any real $f,g \in L^\infty(\Omega, \mu),$ one has
       $$ \|fg - (fg)_\Omega \|_r \leq \|f\|_{p_1} \|g - g_\Omega \|_{q_1} + \|g\|_{p_2} \|f - f_\Omega \|_{q_2}, $$
       where $1 \leq r,p_1,p_2,q_1,q_2 \leq \infty$ and ${1 \over r} = {1 \over p_1} + {1 \over q_1} = {1 \over p_2} + {1 \over q_2}.$
       Additionally, similar statements in rearrangement invariant Banach function spaces shall be presented.
       In fact, Theorem 4.3 asserts for real-valued $f,g \in X$ that 
             $$ \|fg - (fg)_\Omega \|_X \leq \|f\|_{\infty} \|g - g_\Omega \|_X + \|g\|_{\infty} \|f - f_\Omega \|_X,$$
       where $X$ is a rearrangement invariant norm over a Banach function space. 
       
   \section{Preliminaries}
         
    Let $(\Omega, \mathcal{F}, \mu)$ be a measure space.
    Let $\mathcal{L}_0$ denote the algebra of $\mu$-measurable real-valued functions over $\Omega.$ We recall that a linear space $X \subseteq\mathcal{L}_0$
    is a {\it Banach function space} if it is endowed with a norm $\|\cdot\|_X$ that satisfies the ideal property: if $f \in\mathcal{L}_0$
    and $g \in X$ and $|f| \leq |g|$ $\mu$-a.e. then $f \in X$ and $\|f\|_X \leq \|g\|_X$ hold. 
    From here on, we shall suppose that $X$ possesses the Fatou
    property; that is, if $0 \leq f_n \uparrow f$ $\mu$-a.e. and the increasing sequence $\{f_n\}_n$ is norm bounded,
    then $f \in X$ and $\|f_n\|_X \rightarrow \|f\|_X.$ 
    
    We say that two functions $f$ and $g$ are equimeasurable if $|f|$ and $|g|$ have the same distribution
    functions. Then the Banach function space $X$ is {\it rearrangement invariant}
    if for all $f \in X,$  $g \in\mathcal{L}_0,$ where $f$ and $g$ are equimeasurable, we have $g \in X$ and $\|f\|_X = \|g\|_X.$
    
    We notice that if $\mu(\Omega) = 1$ and $\|1\|_X = 1,$ then we have the isometric embeddings
    $L^\infty(\Omega, \mu) \hookrightarrow X \hookrightarrow L^1(\Omega, \mu).$
    
    Well-known examples of rearrangement invariant spaces are 
    the $L^p$-spaces and their certain generalizations, the Lorentz and Orlicz spaces, and the Marczinkiewic spaces.
    
    The {\it decreasing rearrangement} of a $\mu$-measurable $f$ is a nonnegative, decreasing, right-continuous function on $[0, \infty)$ 
    with the same distribution as that of $f.$ We shall denote this function by $f^*.$
    Decreasing rearrangement is not a linear
    but sublinear operator. A simple but useful property of the $*$-operation is the monotone convergence property; i.e. if $|f_n| \uparrow |f|$ $\mu$-a.e. then (2.0) $f^*_n \uparrow f^*$ holds.
    Later, we shall use the important fact that decreasing rearrangements preserve $L^p$ norms. In fact, for all
    $1 \leq p < \infty,$ we have $(|f|^p)^* = (f^*)^p$ hence $\|f\|_p = \|f^*\|_p$ follows,
    and $f^*(0) = \|f\|_\infty.$ Interestingly, the decreasing
    rearrangement is a non-expansive map on $L^p$ spaces, i.e. $\|f^* - g^*\|_p \leq \|f - g\|_p$ \cite{CZR}.

    Let $X$ denote a rearrangement invariant space over the measure space $(\Omega, \mathcal{F}, \mu),$
    which is non-atomic or completely atomic, with all atoms having equal measure.
    The associated space $X'$ (or 
    sometimes called K\"othe dual)
    is the collection of $\mu$-measurable functions $f \in \mathcal{L}_0$ such that 
    $$ \|f\|_{X'} =  \sup \left\{ \int_\Omega fh \: d\mu \colon \|h\|_{X} \leq 1 \right\}$$ is finite.
    In general, the associate space $X'$ is not the Banach space dual of $X.$ But one can find an isometric isomorphism
    from $X'$ onto a closed subspace of the topological dual $X^*,$ which contains sufficiently many linear functionals to get the norm 
    of each element in $X.$
    
    We recall that H\"older's inequality says that if $f$ belongs to $X$ and $g$ belongs to the associate space $X',$ then
     \begin{align}
       \int_\Omega |fg| \: d\mu \leq \int_0^\infty f^*(x) g^*(x) \: dx \leq \|f\|_X \|g\|_{X'},
     \end{align}   
    where the first inequality may be called the Hardy--Littlewood inequality, see \cite{HLP}.
    For general properties of Banach function spaces and proofs of the previous statements,
    we refer the reader to \cite{BS}. 
    
    Let us now choose a non-negative function 
    $f \colon \Omega \rightarrow \mathbb{R}_+$ such that 
    each of its level sets has finite measure:
     $$ \mu(\{ s \colon f(s) > t\}) < \infty.$$
    Then the {\it layer cake} (or {\it wedding cake}) {\it representation} asserts that $f$ can be written as the integral of the characteristic 
    functions of its level sets; that is, 
     $$ f(x) = \int_0^\infty {\bf 1}_{\{s \colon f(s) > t\}}(x) \: dt,$$
    see \cite{LL}, \cite{Simon}.
    Now it is simple to see that if $f$ is a real-valued, then
    $f$ is the integral of differences of characteristic
    functions having disjoint support sets, determined by the positive and negative parts of $f,$ respectively.
    We notice that 
     $$ {\bf 1}_{\{f > t\}}^*= {\bf 1}_{\{f^* > t\}}$$
     for non-negative $f,$ and  
     \begin{align}
       ({\bf 1}_{\{f_+ > t\}}-{\bf 1}_{\{f_- > t\}})^*= {\bf 1}^*_{\{f_+ > t\} \cup \{f_- > t\}} =  {\bf 1}_{\{f^* > t\}}, 
     \end{align}
    where $f_+$ and $f_-$ stand for the positive and negative part of $f.$
     
    The layer cake representation is an efficient tool to prove rearrangement inequalities as the Hardy--Littlewood inequality
    or its generalization, the Riesz's rearrangement inequality, for instance; see e.g. \cite{B}, \cite{HLP}, \cite{LL}, \cite{Simon}.
    While these inequalities are dealing with non-negative functions,
    we shall need some extra care in the next section because of working with real-valued functions.
     
    \section{Main result: a rearrangement inequality}
        
    For simplicity, let us use the notations $a \vee b = \max(a,b)$ and $a \wedge b = \min(a,b)$ whenever
    $a$ and $b$ are reals. 
        
     Let $A \subseteq \Omega$ be a $\mu$-measurable set. We shall use the notation $A^*$ for the support set of
     $\chi_A^*.$ Clearly, $A^*$ is the interval $[0, \mu(A))$ in $\mathbb{R}_+.$ Let $A^c$ denote the complementer set of $A$ in $\Omega.$
     From here on, let $|\cdot|$ stand for the Lebesgue measure of any measurable set of the real line.
     For any $f \colon \Omega \rightarrow \mathbb{R},$ the symmetric (grid-like) function ${\bf 1}_{\mathcal{F}} \colon \Omega \times \Omega \rightarrow \mathbb{R}$  is defined by
     $${\bf 1}_{\mathcal{F}} \colon (x,y) \mapsto {\bf 1}_f(x) \vee {\bf 1}_f(y),$$
     where $x,y \in \Omega$ and ${\bf 1}_f$ denotes the indicator function of supp$(f),$ the support set of $f.$
     Throughout the paper, by the support set of $f$
     we mean the set $\{ f \neq 0\} := \{ x \in \Omega \colon f(x) \neq 0\}.$
       
        For $A, B \subseteq \Omega,$ let us use the notation
               $$ I_{A,B}(f,g,h) :=  \int_A \int_B (f(x)+f(y))(g(x)-g(y))h(y) \: d\mu(x) d\mu(y). $$        
        We start with the following lemma in order to prove Theorem 3.2.        
     
    \begin{lemma} 
         Let $f, g$ and $h$ be real-valued measurable functions over $(\Omega, \mathcal{F}, \mu)$
         such that $g \in L^1(\Omega, \mu)$ and $$\int_\Omega g \: d\mu = 0.$$ 
         Suppose {\rm ess ran} $f \subseteq \{ -1,0,1\},$ and $|h| \leq 1$ $\mu$-a.e. hold. We have
           $$I_{G^c,G}(f,g, h) +  I_{G,G^c}(f,g,h) \leq
              2 \int_{G} \int_{G^c}   |g(y)| {\bf 1}_{\mathcal{F}}(x,y)  {\bf 1}_{\mathcal{H}}(x,y) \: d\mu(x) d\mu(y),$$
         where $G$ is the support of $g.$
    \end{lemma}
    
     \begin{proof}
        First, notice that the sum $f(x)+f(y)$ has constant sign $\mu$-a.e. if $f(y) \neq 0,$ $y$ is fixed, and $x \in \Omega.$
        Since $\int_G 2g \: d\mu = 0,$ we get
              \begin{align*} 
                 I_{G^c,G}&(f,g, h) \\
                 &\leq  \int_{{G^c} \cap \{f \neq 0\}}  \left| \int_{G}  (2-|f(x)+f(y)|) (g(x)-g(y)) h(y) \: d\mu(x) \right| \: d\mu(y) \\
                  & \quad +  \int_{{G^c} \cap \{f = 0\}} \int_{{G} \cap \{f \neq 0\}}   |g(x)| |h(y)| \: d\mu(x) d\mu(y) \\
                 &\leq 
                \int_{G} \int_{{G^c} \cap \{f \neq 0\}} (2-|f(x)+f(y)|) |g(y)| {\bf 1}_{\mathcal{H}}(x,y) \: d\mu(x) d\mu(y) \\
                & \qquad +  \int_{{G} \cap \{f \neq 0\}} \int_{{G^c} \cap \{f = 0\}}   |g(y)| {\bf 1}_{\mathcal{H}}(x,y) \: d\mu(x) d\mu(y),
            \end{align*}
            where we interchanged the order of integration and relabeled the variables in the last step.
            Thus
             \begin{align*}
                I_{G^c,G}(f,g, h) +&  I_{G,G^c}(f,g,h) \\
                \leq& \:  2\int_{G} \int_{{G^c} \cap \{f \neq 0\}} |g(y)| {\bf 1}_{\mathcal{H}}(x,y) \: d\mu(x) \: d\mu(y) \\
                 & \:+ 2 \int_{{G} \cap \{f \neq 0\}} \int_{{G^c} \cap \{f = 0\}}   |g(y)| {\bf 1}_{\mathcal{H}}(x,y) \: d\mu(x)  \: d\mu(y) \\
                 =& \:2 \int_{G} \int_{G^c}   |g(y)| {\bf 1}_{\mathcal{F}}(x,y)  {\bf 1}_{\mathcal{H}}(x,y) \: d\mu(x) d\mu(y), 
              \end{align*}   
         which is what we intended to have.
     \end{proof}
     
    Now we can prove the main theorem of the section.
    
    \begin{thm}
            Let  $(\Omega, \mathcal{F}, \mu)$ be a probability space. Let $f, g \in L^\infty(\Omega, \mu)$ 
            and $h \in L^1(\Omega, \mu)$ be real-valued measurable functions.
            Suppose $$\int_\Omega g \: d\mu = 0.$$ Then
               $$ \iint\limits_{\Omega \times \Omega} (f(x)+f(y))(g(x)-g(y))h(y)  \: d\mu(x) \: d\mu(y)\leq  2 \int_0^\infty f^*(y) g^*(y) h^*(y) \: dy .  $$                  
    \end{thm}
    
     \begin{proof}
       From the layer cake representation, the interchange of order of integrations and (2.2),
       we may assume that the essential ranges of $f$ and $h$ contain only the values $-1, 0, 1.$
       Furthermore, one can find a sequence of simple functions $\{g_n\}_n$ such that $g_n \rightarrow g$ $\mu$-a.e.,
       $|g_n| \uparrow |g|$ $\mu$-a.e. and $\int g_n = 0.$
       The dominated convergence theorem and (2.0) guarantee that we may assume that $g$ is a simple function.
       
       \smallskip
       
       \noindent {\it Step 1:} Let $A$ and $B$ be disjoint measurable subsets of $\Omega.$ Our first step is to tackle the case when $g$ can be written as
       \begin{equation}
        g = a{\bf 1}_A - b {\bf 1}_B,  
       \end{equation}
       where $a\mu(A) = b \mu(B)$ and $0 < b \leq a.$ Let $G = A \cup B$
       be the support of $g,$ and $G^c$ denote its complement in $\Omega,$ as usual.
       Let us split the integral into three parts:
       \begin{align}  
              I_{\Omega, \Omega}(f,g, h) =  I_{G,G^c}(f,g, h) + I_{G^c,G}(f,g, h) + I_{G,G}(f,g, h).
            \end{align}

       {\it Case 1.} Suppose ${\bf 1}^*_f \leq {\bf 1}^*_h;$ that is, $\mu(\{|f| \neq 0\}) \leq \mu(\{|h| \neq 0\}).$ 
   
               From the simple decomposition $G = (G \cap \{f = 0 \}) \cup (G \cap \{f \neq 0 \})$ 
               and the symmetry in $x$ and $y,$ we may infer that
           \begin{align*}
              I_{G,G}(f,g, h) &\leq \iint\limits_{G \cap \{f \neq 0\} \times G \cap \{f \neq 0\}} 2|g(x)-g(y)| {\bf 1}_{\mathcal{H}}(x,y)  \: d\mu(x) d\mu(y) \\
                &\qquad \qquad +  \int_{G \cap \{f \neq 0\}}  \int_{G \cap \{f = 0\}} |g(x)-g(y)| {\bf 1}_{\mathcal{H}}(x,y) \: d\mu(x) d\mu(y)  \\  
                &\qquad \qquad +   \int_{G \cap \{f = 0\}} \int_{G \cap \{f \neq 0\}} |g(x)-g(y)| {\bf 1}_{\mathcal{H}}(x,y) \: d\mu(x) d\mu(y)  \\  
                &= 2 \int_{G \cap \{f \neq 0\}} \int_{G} |g(x)-g(y)| {\bf 1}_{\mathcal{F}}(x,y) {\bf 1}_{\mathcal{H}}(x,y) \: d\mu(x) d\mu(y) \\
                &\leq 2 \int_{G \cap \{f \neq 0\}} \int_{G} |g(x)-g(y)| {\bf 1}_{\mathcal{F}}(x,y) \: d\mu(x) d\mu(y) \\
                &=: \Psi_1(f,g,h), 
            \end{align*} 
             since ${\bf 1}_{\mathcal{F}}(x,y) {\bf 1}_{\mathcal{H}}(x,y)  \leq {\bf 1}_{\mathcal{F}}(x,y)$ holds by
            our assumption on $f$ and $h.$ For simplicity, let $F$ denote the support set of $f.$ 
             Then
             \begin{align*}
                \Psi_1(f,g,h) = 2(a+b) \mu(A \cap F) \mu(B) + 2(a+b)\mu(B \cap F)  \mu(A).
             \end{align*}     
            Moreover, notice that
              \begin{equation}              
              \int_{G^c} {\bf 1}_{\mathcal{F}}(x,y) \: d\mu(x) = \begin{cases}
                                                                   \mu(G^c) & \mbox{ if } y \in F \\
                                                                   \mu(G^c \cap F) & \mbox{ if } y \not \in F.
                                                             \end{cases}
              \end{equation}
            Thus, from Lemma 3.1 and (3.3),
             \begin{align*}
               I_{G^c,G}(f,g, h) +  I_{G,G^c}&(f,g, h) \\
                                               &\leq 2 \int_{G}  |g(y)|  \left( \int_{G^c} {\bf 1}_{\mathcal{F}}(x,y) \: d\mu(x) \right) d\mu(y) \\
                                               &= 2a \mu(A \cap F) \mu(G^c)  + 2a \mu(A \setminus F) \mu(G^c \cap F) \\
                                                & \quad + 2b \mu(B \cap F) \mu(G^c) + 2b \mu(B \setminus F) \mu(G^c \cap F) =: \Psi_2(f,g,h).
             \end{align*}                                                                 
            Obviously, $a\mu(A) = b\mu(B)$ and $\mu(A) + \mu(G^c) + \mu(B) = 1.$ Hence, with a little computation, we get
             \begin{align*} 
                     \Psi_1(f,g,h) &+ \Psi_2(f,g,h) \\
                             &= 2a \mu(A \cap F) + 2b\mu(B \cap F) + 2\mu(G^c \cap F)(a\mu(A \setminus F) + b\mu(B \setminus F)).
             \end{align*}
             Furthermore, we have the following estimates of the previous sum.
             First, 
             \begin{align}
                 \Psi_1(f,g,h) + \Psi_2(f,g,h) \leq 2a \mu(A \cap F) + 2a\mu(B \cap F) + 2a\mu(G^c \cap F) = 2a \mu(F).
             \end{align}
             Secondly,
               \begin{align}
                \begin{split}
                 \Psi_1(f,g,h) + \Psi_2(f,g,h) &\leq 2a \mu(A \cap F) + 2b\mu(B \cap F) + 2(a\mu(A \setminus F) + b\mu(B \setminus F))  \\ 
                   &= 2a \mu(A) + 2b \mu(B).
                \end{split}   
             \end{align}
             And lastly, we claim
                 \begin{align}
                 \Psi_1(f,g,h) + \Psi_2(f,g,h) \leq 2a \mu(A) + 2b (\mu(F) - \mu(A)).
             \end{align}
             We can prove (3.6) if we show that 
             \begin{align*}
               2a \mu(A \cap F) + 2\mu(G^c \cap F)(a\mu(A \setminus F) &+ b\mu(B \setminus F)) \\
                                                 &\leq 2a \mu(A) + 2b (\mu(F \cap (A\cup G^c)) - \mu(A)).             
             \end{align*}
              Or equivalently, 
              \begin{align*}
                 2\mu(G^c \cap F)&(a\mu(A \setminus F) + b\mu(B \setminus F)) \leq  2(a-b)\mu(A \setminus F)  + 2b\mu(G^c \cap F),
              \end{align*}
              which readily holds, since $2(a-b)\mu(A \setminus F) \geq 2(a-b)\mu(G^c \cap F)\mu(A \setminus F)$ and 
              $2b\mu(G^c \cap F) \geq 2b\mu(G^c \cap F)(\mu(A \setminus F)+\mu(B \setminus F))$ as $A \setminus F$ and $B \setminus F$ are disjoints. Hence our claim 
              is established.
              
              From (3.4), (3.5) and (3.6), we conclude that
              \begin{align*}
                \Psi_1(f,g,h) +& \Psi_2(f,g,h) \\
                  &\leq 2a \mu(F) \wedge (2a \mu(A) + 2b \mu(B))  \wedge (2a \mu(A) + 2b (\mu(F) - \mu(A))) \\
                  &\leq 2a (\mu(F)  \wedge \mu(A)) + 2b (0 \vee (\mu(F) - \mu(A)))  \wedge \mu(B).
              \end{align*}
             However, a careful look upon the right-hand side shows that it equals the integral 
             \begin{align*}
              2\int_0^\infty f^*(y) g^*(y) h^*(y) \: dy &= 2a|A^* \cap F^*| + 2b| (F^* \setminus A^*)^* \cap B^*|.
             \end{align*}             
             Hence, the estimates of the applied decomposition (3.2) complete the proof in this case.
             
            \noindent {\it Case 2.} Lastly, suppose ${\bf 1}^*_f \geq {\bf 1}^*_h.$ 
              Now we clearly have
             $$ I_{G,G}(f,g,h) \leq 2 \int_{G \cap \{h \neq 0\}} \int_{G} |g(x)-g(y)|\: d\mu(x) d\mu(y).$$
             The remaining part of (3.2) can be estimated by Lemma 3.1, providing an upper bound which is symmetric in $f$ and $h.$
             Hence, if we interchange the role of $f$ and $h$ in $\Psi_1(f,g,h)$ and $\Psi_2(f,g,h),$ we can
             finish the proof by calculations previously done in Case 1. Thus we get the proof under the assumption (3.1).
              \smallskip      
                        
              {\it Step 2:} In the general case, we can decompose the simple function $g = \sum_{i=1}^m a_i {\bf 1}_{E_i} - \sum_{i=1}^n b_i {\bf 1}_{F_i}$ into 
              sums of functions of zero means used in Step 1 as (3.1). 
              Suppose that $a_m > a_{m-1} > \hdots > a_1 > 0 > b_1 > \hdots > b_n$ and the sets $E_i, F_i$ are pairwise disjoints.
              Then we need to section $g$ into pair of horizontal blocks such that each pair has zero mean.
              
              First, if 
               $$ a_1\mu(\cup_{i=1}^m E_i) \leq |b_1| \mu(\cup_{i=1}^n F_i),$$
              let $\tilde{a}_1 = a,$ $\tilde{b}_1 = a_1 \mu(\cup_{i=1}^m E_i)/\mu(\cup_{i=1}^n F_i).$ 
              Otherwise, let $\tilde{a}_1 = b_1 \mu(\cup_{i=1}^n F_i)/\mu(\cup_{i=1}^m E_i),$ $\tilde{b}_1 = b.$
              Set
                 $$ g_1 = \tilde{a}_1 {\bf 1}_{\bigcup_{i=1}^m E_i} - \tilde{b}_1 {\bf 1}_{\cup_{i=1}^n F_i}. $$   
              Clearly, $g_1$ has zero mean value over $\Omega$ and supp$(g) = \mbox{supp}(g_1) \cup E_1$ or supp$(g_1) \cup F_1.$ Next, let us repeat this construction with $g - g_1$ to get $\tilde{a}_2, \tilde{b}_2$ and the function $g_2$ and so on
              until we arrive at the zero function $\mu-$a.e. Hence, we get a  decomposition 
                             $$ g  = \sum_{i=1}^K g_i  $$
                and the supports 
              form a decreasing sequence $\mbox{supp } g_1 \supseteq \mbox{ supp } g_{2} \supseteq \hdots \supseteq \mbox{ supp } g_K .$ 
              Since the decreasing rearrangement of simple functions may be viewed as sliding the blocks
              in each horizontal layer to build a single larger block, it follows
              $$ g^*  = \sum_{i=1}^K g_i^*.$$
              Since the left-hand side of the desired inequality is linear in $g,$ the proof now follows
              in the general case as well. 
         \end{proof}
         
    \section{Applications}
    
    We can now present a Leibniz-type rule for mean oscillations of functions.
    The result first appeared in our earlier paper \cite{LMia} (see \cite[Theorem 2.6]{BeL}, \cite[Theorem 5.1]{LJmaa} in particular cases as well).
    However, we think that the next proof is transparent and considerably simpler, being a straightforward corollary of Theorem 3.2.
    
    \begin{thm}
    Let $(\Omega, \mathcal{F}, \mu)$ be a probability space. For any real $f,g \in L^\infty(\Omega, \mu),$ one has
    $$ \|fg - (fg)_\Omega \|_r \leq \|f\|_{p_1} \|g - g_\Omega \|_{q_1} + \|g\|_{p_2} \|f - f_\Omega \|_{q_2}, $$
    where $1 \leq r,p_1,p_2,q_1,q_2 \leq \infty$ and ${1 \over r} = {1 \over p_1} + {1 \over q_1} = {1 \over p_2} + {1 \over q_2}.$
    \end{thm}
   
    \begin{proof}
         First, we notice that the identity
          \begin{equation}
           f(x)g(x) - f(y)g(y) = {1 \over 2} (f(x)+f(y))(g(x)-g(y)) 
                                   + {1 \over 2} (f(x)-f(y))(g(x)+g(y))           
          \end{equation}
         holds for all $x, y \in \Omega.$
         
         From the duality of $L^p$-spaces, one can find a real $h \in L^{r'}(\Omega, \mu)$ such that $\|h\|_{r'} = 1$ and
         $$ \|fg - (fg)_\Omega\|_r = \int_\Omega \left(fg - \int_\Omega fg \: d\mu \right)h \: d\mu.$$
         Relying upon Theorem 3.2 and the identity (4.1), we get
          \begin{align*}
             \int_\Omega (fg - (fg)_\Omega)h \: d\mu &= \iint\limits_{\Omega \times \Omega} (f(x)g(x) - f(y)g(y))h(y) \: d\mu(x) \: d\mu(y) \\
                                                        &\leq  \int_0^\infty f^*(y) \left(g - \int_\Omega g \: d\mu \right)^*(y) h^*(y) \: dy \\
                                                        & \qquad +   \int_0^\infty g^*(y)  \left(f - \int_\Omega f \: d\mu \right)^*(y) h^*(y) \: dy, \\
          & \hspace{-4.2cm} \mbox{and since } h \mbox{ and } h^* \mbox{ are equimeasurable: }  \|h\|_{r'} = \|h^*\|_{r'} = 1, \mbox{ hence }    \\
            &\leq   \|f^* (g - g_\Omega)^* \|_{L^r[0,\infty)} +   \|g^* (f - f_\Omega)^* \|_{L^r[0,\infty)}. 
          \end{align*}
        Applying H\"older's inequality and using again that any function and its decreasing
        rearrangement have the same $L^p$ norms, the proof is complete.
    \end{proof}
    
    \begin{remark}
      The next reasoning provides a heuristic approach to the decomposition (4.1). Let $(\Omega, \mathcal{F}, \mu)$ be a probability measure space. We can define a first-order differential calculus over
      $L^\infty(\Omega, \mu).$ Any real $m \in L^\infty(\Omega, \mu)$ gives a left and right multiplication on the real product space
      $L^2(\Omega \times \Omega, \mu \otimes \mu) $ by
       $$ (mf)(x,y) = m(x)f(x,y) \mbox{ and } (fm)(x,y) = m(y)f(x,y).$$ 
       Let us consider the derivation $\partial \colon L^\infty(\Omega, \mu) \rightarrow L^2(\Omega, \mu) \otimes L^2(\Omega, \mu)$
       by
        $$ (\partial f)(x,y) = f(x) - f(y),$$
        which clearly satisfies the Leibniz rule 
        $$ \partial(fg) = f (\partial g) + (\partial f) g. $$ 
        The map $\partial$ has a natural linear extension to $L^2(\Omega).$ We recall that the integral formula for the variance
         $$ \|f - f_\Omega\|_2^2 = {1 \over 2} \iint_{\Omega \times \Omega} (f(x)-f(y))^2 \: d\mu(x) \: d\mu(y)$$ holds.
         Now it is simple to check that $\|\partial f \|_2^2 = 2\|f - f_\Omega\|_2^2$ and
         $$ -2(f - f_\Omega) = \partial^* \partial f$$ hold. Thus,
         $ -2(fg - (fg)_\Omega) =  \partial^* (f \partial g) + \partial^* ((\partial f) g).$
         Moreover,
         $$  -\partial^* (f \partial g)(y) =\int_\Omega (f(x)+f(y))(g(x) - g(y)) \: d\mu(x)$$
        and 
         $$  -\partial^* ((\partial f )g )(y) =  \int_\Omega  (g(x)+g(y))(f(x) - f(y)) \: d\mu(x),$$
        for almost every $y \in \Omega,$  which lead to the decomposition appeared in the proof of Theorem 4.1. 
    \end{remark}

    A very similar argument provides us with a Leibniz-type inequality in rearrangement invariant Banach function spaces.
       
     \begin{thm}
            Let $(\Omega, \mathcal{F}, \mu)$ be a probability space which is non-atomic or completely atomic, with
            all atoms having equal measure, 
            and suppose $X$ is a rearrangement invariant function space over it.
            If $f,g \in X$ are real-valued and bounded, then 
             $$ \|fg - (fg)_\Omega \|_X \leq \|f\|_{\infty} \|g - g_\Omega \|_X + \|g\|_{\infty} \|f - f_\Omega \|_X. $$
             Furthermore, if $f$ belongs to $X$ and $g$ belongs to the associate space $X',$ then
              $$ \|fg - (fg)_\Omega \|_1 \leq \|f\|_{X} \|g - g_\Omega \|_{X'} + \|g\|_{X'} \|f - f_\Omega \|_{X}. $$
    \end{thm}
    
    \begin{proof}
    We recall that the Lorentz--Luxemburg theorem \cite[Theorem 2.7]{BS} asserts that $X$ coincides with its second
    associate space $X''$ and $\|f\|_X = \|f\|_{X''},$ hence
    $$  \|f\|_X =  \sup \left\{ \int_\Omega fh \: d\mu \colon \|h\|_{X'} \leq 1 \right\}.$$
    Pick any $h \in X'$ such that $\|h\|_{X'} \leq 1.$ We get from Theorem 3.2 and (4.1)
     \begin{align*}
             \int_\Omega (fg - (fg)_\Omega)h \: d\mu &= \iint\limits_{\Omega \times \Omega} (f(x)g(x) - f(y)g(y))h(y) \: d\mu(x) \: d\mu(y) \\
                                                        &\leq f^*(0) \int_0^\infty \left(g - \int_\Omega g \: d\mu \right)^*(y) h^*(y) \: dy \\
                                                        & \qquad + g^*(0)  \int_0^\infty  \left(f - \int_\Omega f \: d\mu \right)^*(y) h^*(y) \: dy.
          \end{align*}
     Applying decreasing rearrangements to describe the norm $X,$ see \cite[Corollary 4.4]{BS}, we obtain
      \begin{align*}
             \int_\Omega (fg - (fg)_\Omega)h \: d\mu \leq \|f\|_{\infty} \|g - g_\Omega \|_X + \|g\|_{\infty} \|f - f_\Omega \|_X.
     \end{align*}             
     Taking the supremum of the left-hand side over $h$ such that $\|h\|_{X'} \leq 1,$ the proof is complete.
     
     The rest of the statement can be proved by a same argument and (2.1), therefore it is left to the reader.
     \end{proof}
    
     \begin{remark}
      The statement of the previous theorem remains valid in rearrangement invariant spaces over
      arbitrary probability measure space. In fact, one can extend Theorem 4.3 to this general case by the {\it method of retracts,} described in \cite[p. 54]{BS},
      which enables us to embed any probability space into a non-atomic one.
     \end{remark}
   
    \begin{remark}
       We note that any rearrangement invariant norm $X$ generates further rearrangement invariant norms via the expression
       $\||f|^p\|_X^{1/p},$ for $1 \leq p < \infty.$ Hence one may find a possible extension of Theorem 4.3 in the spirit 
       of Theorem 4.1. We left the proof of this direction to the interested reader. 
     \end{remark}


\begin{thebibliography}{20}
          \bibitem{BS} {\sc C. Bennett and R. Sharpley}, {\it Interpolation of Operators}, Academic Press Inc., 1988.
          \bibitem{BeL}{\sc \'A. Besenyei and Z. L\'eka}, Leibniz seminorms in probability spaces, {\it J. Math. Anal. Appl.}, {\bf 429} (2015), 1178--1189.                   
          \bibitem{Bhatia} {\sc R. Bhatia}, {\it Matrix analysis}, Springer--Verlag New York, 1997. 
          \bibitem{BH} {\sc N. Bouleau and F. Hirsch}, {\it Dirichlet Forms and Analysis on Wiener space}, de Gruyter Studies in Mathematics, De Gruyter, Berlin, 1991. 
          \bibitem{B}  {\sc A. Burchard}, A short course on rearrangement inequalities, preprint.
          \bibitem{CS} {\sc F. Cipriani and J.-L. Sauvageot}, Derivations as square roots of Dirichlet forms, {\it J. Func. Anal.}, {\bf 201} (2003), 78--120.
          \bibitem{CZR}{\sc J.A. Crowe, J.A. Zweibel and P.C. Rosenblum}, Rearrangements of functions, {\it J. Func. Anal.}, (1986) {\bf 66}, 432--438.
          \bibitem{Fuk}{\sc M. Fukushima}, {\it Dirichlet Forms and Markov Processes}, North Holland Mathematical Library, 1980.
          \bibitem{G}  {\sc L. Grafakos}, {\it Modern Fourier Analysis}, Graduate Texts in Mathematics, Springer, New York, 2014.
          \bibitem{HLP} {\sc G.H. Hardy, J.E. Littlewood and G. P\'olya}, {\it Inequalities}, Cambridge University Press, 1952. 
          \bibitem{LJmaa} {\sc Z. L\'eka}, Symmetric seminorms and the Leibniz property, {\it J. Math. Anal. Appl.}, {\bf 452} (2017), 708--725.
          \bibitem{LMia}  {\sc Z. L\'eka}, On the Leibniz rule for random variables, to appear in {\it  Math. Inequal. Appl.}
          \bibitem{LL}  {\sc E. Lieb and M. Loss}, {\it Analysis}, Graduate Texts in Mathematics, Amer. Math. Soc., 1997.
          \bibitem{MS}  {\sc C. Muscalu and W. Schlag}, {\it Classical and Multilinear Harmonic Analysis}, Vol. 2., Cambridge University Press, 2013.
          \bibitem{R0}  {\sc M.A. Rieffel}, Metrics on state spaces, {\it Doc. Math.}, {\bf 4} (1999), 559--600.
         \bibitem{RL}  {\sc M.A. Rieffel}, Leibniz seminorms for ``matrix algebras converge to the sphere``, {\it Quanta of Maths} {\bf 11}, Amer. Math. Soc., Providence, RI, 2010,  543--578.
         \bibitem{R1}  {\sc M.A. Rieffel}, Non-commutative resistance networks, {\it SIGMA Symmetry Integrability Geom. Methods Appl.},{ \bf 10} (2014), 2259--2274.     
         \bibitem{R2}  {\sc M.A. Rieffel}, Standard deviation is a strongly Leibniz seminorm, {\it New York J. Math.}, {\bf 20} (2014), 35--56.
         \bibitem{Simon} {\sc B. Simon}, {\it  Convexity: An Analytic Viewpoint}, Cambridge University Press, 2011.     
     \end{thebibliography}
\end{document}